\documentclass[11pt]{article}

\usepackage{amsthm}
\usepackage{amsmath}
\usepackage{amssymb}
\usepackage{enumerate}
\usepackage{fullpage}
\usepackage{graphicx}
\usepackage{caption}
\usepackage{mathtools}

\newtheorem{theorem}{Theorem}
\newtheorem{corollary}[theorem]{Corollary}
\newtheorem{lemma}[theorem]{Lemma}

\newtheorem{claim}[theorem]{Claim}

\newcommand{\var}{\textbf{Var}}
\newcommand{\Var}{\textbf{Var}}	
\newcommand{\cov}{\textbf{Cov}}

\newcommand{\dt}{~dt}
\newcommand{\Z}{\mathbb{Z}}
\newcommand{\ffact}[2]{(#1)_{#2}}
\DeclareMathOperator*{\E}{\mathbb{E}}

\newcommand*{\defeq}{\stackrel{\text{def}}{=}}

\newcommand{\xx}{{\mathbf{A}, \mathbf{A'}}}

\title{A local central limit theorem for triangles in a random graph}
\author{Justin Gilmer\thanks{Department of Mathematics,
Rutgers University, Piscataway, NJ, USA. \texttt{jmgilmer@math.rutgers.edu}.
Research supported by NSF CCF-083727 and CCF-1218711.}
\and
Swastik Kopparty\thanks{Department of Mathematics \& Department of Computer Science, Rutgers University, Piscataway, NJ, USA. {\tt swastik.kopparty@rutgers.edu}. Research supported in part by a Sloan Fellowship and NSF CCF-1253886.} }

\begin{document}

\maketitle

\begin{abstract}
In this paper, we prove a local limit theorem for the distribution of the number of triangles in the Erdos-Renyi random graph $G(n,p)$, where $p \in (0,1)$ is a fixed constant. Our proof is based on bounding the characteristic function $\psi(t)$ of the number of triangles, and uses several different conditioning arguments for handling different ranges of $t$.
\end{abstract}

\section{Introduction}
We will work with the Erdos-Renyi random graph $G(n,p)$.
Recall that $G(n,p)$ is the random undirected graph
$G$ on $n$ vertices sampled by including
each of the ${n \choose 2}$ possible edges into $G$ independently
with probability $p$.
Let $S_n$ be the random variable equal to the number of triangles in $G(n,p)$.
Let $\mu_n = \E[S_n] = p^3 {n \choose 3}$ and $\sigma_n = \sqrt{\Var[S_n]} = \Theta(n^2)$ (see the Appendix for an exact calculation of $\sigma_n$).
Our main result (Theorem~\ref{thm:main}) states that if $p$ is a fixed constant in $(0,1)$, then the distribution of $S_n$ is {\em pointwise} approximated by
a discrete Gaussian distribution:
\begin{align}
\label{counteq}
\Pr[S_n = k ]  =   \frac{1}{\sqrt{2 \pi} \sigma_n} e^{-\left((k - \mu_n)/\sigma_n\right)^2/2} \pm o(1/n^2).
\end{align}
Thus, for every $k \in \mu_n \pm O(n^2)$, we determine the probability that $G(n,p)$ has exactly $k$ triangles, up to a $(1+o(1))$ multiplicative factor.

\subsection{Central Limit Theorems}
The study of random graphs has over 50 years of history, and understanding the distribution of subgraph counts has long been a central question in the theory. When the edge probability $p$ is a fixed constant in $(0,1)$, there is a classical central limit theorem for the triangle count $S_n$ (as well as for other connected subgraphs). This theorem says that for fixed constants $a,b$:
\[\left|\Pr\left[ a \leq (S_n-\mu_n)/\sigma_n \leq b\right] - \int\limits_a^b \frac{1}{\sqrt{2\pi}} e^{-t^2/2}\dt \right| = o(1),\]
(in other words, $(S_n -\mu_n) / \sigma_n$ converges in distribution to the standard Gaussian distribution).
There are several proofs of the central limit theorem for subgraph counts, as well as some vast generalizations, known today.

The original proofs of the central limit theorem for triangle counts (and general subgraph counts) used the method of moments.
This method is based on the fact for all distributions that are uniquely determined by their moments, the convergence of the moments of a sequence of random variables to the moments of the distribution implies convergence in distribution. Application of the moment method to subgraph statistics goes back to Erdos and Renyi's original paper~\cite{erdos1960evolution}. There were several papers in the 1980's (see \cite{karonski1983number} and \cite{karonski1984balanced}) that used the moment method to understand, under increasingly general assumptions, when normalized subgraph counts converge in distribution to the Gaussian distribution. This line of work culminated with a paper by Ruci\'nski~\cite{rucinski1988small} who completely characterized when normalized subgraph counts converge in distribution to the Gaussian distribution. 

There are several other approaches to the central limit theorem for triangle counts (and general subgraph counts).
Using Stein's method~\cite{stein1971dependent}, Barbour, Karo\'nski and Ruci\'nski~\cite{barbour1989central} obtained strong quantitative bounds on the error in the central limit theorem for subgraph counts. A technique from the asymptotic theory of statistics, known as $U$-statistics, was applied by Nowicki and Wierman~\cite{nowicki1988subgraph} to obtain a central limit theorem for subgraph counts, although, in a slightly less general setting than the theorem of Ruci\'nski. Janson \cite{janson1992orthogonal} used a similar method with several applications, including central limit theorems for the joint distribution of various graph statistics. 
None of these techniques, however, seem to be quantitatively strong enough to estimate the point probability mass of the triangle/subgraph counts when the edge probability $p$ is a constant.

\subsection{Poisson Convergence}
When the edge probability $p$ is small enough (for example, $p \approx c/n$ for triangles), then there are classical results
that give good estimates for $\Pr[S_n = k]$. In this regime, the distribution of the subgraph count $S_n$ itself (i.e., without normalization) converges in distribution (and hence pointwise) to a Poisson random variable. Some of the work dedicated to understanding this probability regime goes back to the original paper of Erdos and Renyi \cite{erdos1960evolution} who studied the distribution of counts of trees and cycles using the method of moments. Using  Chen's \cite{chen1975poisson} generalization of Stein's method to the Poisson setting,  Barbour \cite{barbour1982poisson} showed Poisson convergence for general subgraph counts. In the Poisson setting, the probability mass is concentrated in an interval of constant size and thus all results are ``local'' in the sense that they bound the point probability mass of these random variables.
   
 For slightly larger $p \in [n^{-1}, O(n^{-(1/2)})]$ (this is the range of $p$ where $\sigma_n = \Theta(\mu_n)$), R{\"o}llin and Ross~\cite{rollin2010local} showed that the probability mass function for triangle counts ($S_n$) is close in the $\ell_\infty$ and total variation metrics to the probability mass function of a translated Poisson distribution (and hence a discrete Gaussian distribution), and asked whether a similar local limit law holds for larger $p$ (See Remark 4.5 of that paper). Our result gives such a law for constant $p \in (0,1)$ for the $\ell_\infty$ metric. 
   
\subsection{Subgraph counts mod $q$}
Some more recent works studied the distribution of subgraph counts mod $q$.
For example, Loebl, Matousek and Pangrac~\cite{loebl2004triangles} studied the distribution of
$S_n$ mod $q$ in $G(n,1/2)$. They showed that when $q \in (\omega(1), O(\log^{1/3} n))$, then for 
every $a \in \Z_q$, the probability that $S_n \equiv a$ mod $q$ equals $(1 + o(1)) \cdot \frac{1}{q}$.
Kolaitis and Kopparty~\cite{kolaitis2013random} also studied this problem in $G(n,p)$ for fixed $p \in (0,1)$.
They showed that for every constant $q$, and every $a \in \Z_q$,
the probability that $S_n \equiv a$ mod $q$ equals $(1 + \exp(-n)) \cdot \frac{1}{q}$. This latter result
also generalizes to all connected subgraph counts, and to multidimensional versions for the joint distribution of
all connected subgraph counts simultaneously. DeMarco, Kahn and Redlich~\cite{demarco2014modular} extended these results of~\cite{kolaitis2013random}
to determine the distribution of subgraph counts mod $q$ in $G(n,p)$ for all $p$.
Many of these works use conditioning arguments that are similar to those used here.

\subsection{Our result}
The above lines of work:
\begin{enumerate}
\item the central limit theorem for triangle counts in $G(n, p)$ with $p$ constant, 
\item the Poisson local limit theorem for triangle counts in $G(n,p)$ with $p$ close to $n^{-1}$,
\item the uniform distribution of triangle counts mod $q$ in $G(n,p)$ with $p$ constant,
\end{enumerate}
all strongly suggest the truth of our main theorem (Theorem~\ref{thm:main}): there is a local discrete Gaussian limit law for
triangle counts in $G(n,p)$ with $p$ constant.

The high level structure of our proof follows the basic Fourier analytic strategy behind the classical local limit theorem for the sums of i.i.d. integer valued random variables. To show that the distribution of $(S_n - \mu_n)/\sigma_n$ is close pointwise to the discrete Gaussian distribution (as in equation~\eqref{counteq}), it suffices
to show that their characteristic functions (Fourier transforms) are close in $L_1$ distance. Specifically, if we define $\psi_n(t) = \E[e^{i t (S_n- \mu_n)/\sigma_n}]$, we need to show that:
$$ \int_{-\pi \sigma_n}^{\pi \sigma_n} |\psi_n(t) - e^{-t^2/2} | \dt = o(1).$$ The central limit theorem for triangle counts can be used to bound
the above integral in the range $(-A, A)$ for any large constant $A$. 
By choosing $A$ large enough, we can bound $\int_{A < |t| < \pi \sigma_n} |e^{-t^2/2} | \dt$ by an arbitrarily small constant. We are thus reduced to showing that
$\int_{A < |t| < \pi \sigma_n} |\psi_n(t)|\dt = o(1)$. We achieve this using two different arguments. For $A < |t| < n^{0.55}$, we show that
$|\psi_n(t)| < \frac{1}{t^{1+\delta}}$ using a conditioning argument, where we first reveal the edges in a set $F \subseteq { [n] \choose 2 }$, and count
triangles according to how many edges they have in $F$. For $n^{0.55} < |t| < \pi \sigma_n$, we show that $|\psi_n(t)|$ is superpolynomially small in $t$
by another conditioning argument, where we partition the vertex set $[n]$ into two sets $U$ and $V$, first expose all the edges within $V$,
and then consider the increase to the total number of triangles that occurs when we expose the remaining edges.

We conjecture that a similar local discrete Gaussian limit law should hold for the number of copies of any fixed connected graph $H$ in $G(n,p)$, 
for any $p$ that lies above the threshold probability for appearance of $H$. It would also be interesting to understand the joint distribution
of subgraph counts in $G(n,p)$ for several fixed connected graphs. It seems like there are many interesting questions here and much to be
investigated.



\section{Notation and Preliminaries}

Let $[n]$ denote the set $\{1,2,\ldots, n\}$. 
For each positive integer $n$ let $K_n$ be the complete graph on the vertex set $[n]$.
The Erdos-Renyi random graph $G(n,p)$ is the graph $G$ with vertex set $[n]$, where for each $e \in {[n] \choose 2}$, 
the edge $e$ is present in $G$ independently with probability $p$. For $e \in {[n] \choose 2}$, let
$X_e$ denote the indicator for the event that edge $e$ is present in $G$. For $E \subseteq {[n] \choose 2}$, we will let $\{0,1\}^E$ denote the set of $\{0,1\}$-vectors indexed by $E$. Likewise $X_E \in \{0,1\}^E$ will be the random vector for which the value on coordinate $e$ is the random variable $X_e$. 

For the rest of the paper $p \in (0,1)$ will be a universal fixed constant. All asymptotic notation will hide constants which may depend on $p$. We will use $S_n$ to denote the number of triangles in $G(n,p)$ (thus $S_n \in [0, {n \choose 3}]$). The mean of $S_n$ is $p^3\binom{n}{3}$ and the variance (see Appendix) is $\sigma_n^2 = \Theta(n^4)$. We let $R_n$ denote the normalized triangle count, $R_n \defeq \frac{S_n - p^3\binom{n}{3}}{\sigma_n}$. 

{\bf Fourier inversion formula for lattices:} Let $Y$ be a random variable that has support contained in the (shifted) discrete lattice $\mathcal{L} \defeq \frac{1}{b} (\mathbb{Z} - a)$ for real numbers $a, b$. Let $\psi(t) \defeq \E[e^{itY}]$ be the characteristic function of $Y$. Then for all $y \in \mathcal{L}$ it holds that
 \begin{equation}\label{eq:inversion}
 \Pr(Y = y) = \frac{1}{2\pi b} \int\limits_{-\pi b }^{\pi b} e^{-ity}\psi(t)\dt.
 \end{equation} 

Throughout the paper, for real numbers $x$ we will use $\|x\|$ to denote the distance from $x$ to the nearest integer. 
We will often apply the following easy bound. 


\begin{lemma}
\label{lem:coslem}
Let $B$ be a Bernoulli random variable that is 1 with probability $p$.
Then:
$$| \E_B [ e^{i \theta B} ] |  \leq 1 - 8p(1-p) \cdot \left\| \theta/2\pi \right\|^2 .$$
\end{lemma}
\begin{proof}

  Without loss of generality, we may assume that $\theta \in [-\pi, \pi]$.
  We first state two elementary inequalities:
    \begin{equation} \label{eq:cos}
    \cos(t) \leq 1 - 8 \cdot \| t/2\pi \|^2 \qquad \text{(for $t \in [-\pi, \pi]$)}
    \end{equation}
    and 
    \begin{equation} \label{eq:sqrt}
    \sqrt{1-t} \leq 1 - t/2 \qquad \text{(for $t \leq 1$).}
    \end{equation}
    Then we have the following,
\begin{align*}
 |E[e^{i \theta b}] |  & = |p + (1-p)e^{i \theta} | \\
  & = \sqrt{  p^2 + (1-p)^2 + 2p(1-p) \cos (\theta) } \\
 & \leq \sqrt{ p^2 + (1-p)^2 + 2p (1-p) \left( 1 - 8 \cdot \| \theta/2\pi \|^2 \right) } \qquad \text{(applying \eqref{eq:cos})}  \\
 & = \sqrt {   1  -  16 p (1-p) \| \theta/2\pi \|^2  }  \\
 & \leq  1  - 8 p(1-p) \| \theta/2\pi \|^2 \qquad \text{(applying \eqref{eq:sqrt})}.
\end{align*}

\end{proof}

\section{Main Result}



We now give a formal statement of our main result.

   
   
   \begin{theorem}[Local limit law for triangles in $G(n,p)$] \label{thm:main}
    Let
  \[p_n(x) = \Pr(R_n = x)  \text{\ \  for } x \in \mathbb{L}_n = \left\lbrace\frac{k - p^3\binom{n}{3}}{\sigma_n} : k \in \mathbb{Z}\right\rbrace\]
  and 
  \[\mathcal{N}(x) = \frac{1}{\sqrt{2\pi}} e^{-x^2/2} \text{\ \ for } x \in (-\infty,\infty).\]
  Then as $n \rightarrow \infty$,
  \[\sup\limits_{x \in \mathbb{L}_n} \left| \sigma_n p_n(x) - \mathcal{N}(x)\right| \rightarrow 0.\]
  \end{theorem}

Equivalently, we have that for all $n$, for all $k \in \mathbb Z$,
$$ \Pr[S_n = k ] =  \frac{1}{\sigma_n} \cdot \mathcal{N}\left(\frac{k - p^3 \cdot {n \choose 3}}{\sigma_n} \right) + o\left(\frac{1}{n^2}\right),$$
(where the $o(1)$ term goes to $0$ as $n \to \infty$, uniformly in $k$).

   \begin{proof}
   Let $\psi_n(t) = \mathbb{E}[e^{i t R_n}]$. Then the Fourier inversion formula for lattices (equation \ref{eq:inversion}) gives us
   \[ \sigma_n p_n(x) = \frac{1}{2\pi} \int\limits_{-\pi \sigma_n}^{\pi \sigma_n} e^{-itx}\psi_n(t)\dt.\]
   The standard Fourier inversion formula (for $\mathbb R$), along with the well known formula for the Fourier transform of $\mathcal{N}$, gives us:
   \[\mathcal{N}(x) = \frac{1}{2\pi} \int\limits_{-\infty}^{\infty} e^{-itx}e^{- t^2 / 2} \dt.\]
   Therefore,
   \begin{equation}\label{eq:eq0} \left| \sigma_n p_n(x) - \mathcal{N}(x) \right| \leq \int\limits_{-\pi \sigma_n}^{\pi \sigma_n} |\psi_n(t) - e^{-t^2/2}|\dt + 2\int\limits_{\pi \sigma_n}^{\infty} e^{-t^2/2}\dt
   \end{equation}
   The second term goes to zero as $n$ tends to infinity. Thus, it suffices to show that
   \begin{equation}
\label{tobound}
      \int\limits_{-\pi \sigma_n}^{\pi \sigma_n} \left|\psi_n(t) - e^{- t^2/2}\right|\dt 
   \end{equation}
tends to $0$.

   Let $A>0$ be a large constant to be determined later. We divide the integral into three regions:
   \begin{itemize}
   \item $R_1 = (-A,A)$
   \item $R_2 = (-n^{0.55}, -A) \cup (A, n^{0.55})$
   \item $R_3 = (-\pi \sigma_n, -n^{0.55}) \cup (n^{0.55}, \pi \sigma_n)$
   \end{itemize}
The following three lemmas will help us bound the integral of $ \left|\psi_n(t) - e^{- t^2/2}\right|$ in these three regions.
\begin{lemma} \label{lem:tconst}
Let $A$ be a fixed positive real number. Then
\[ \int\limits_{-A}^{A} \left| \psi_n(t) - e^{- t^2/2}\right|\dt \rightarrow 0\]
as $n \rightarrow \infty$.
\end{lemma}

\begin{lemma} \label{lem:smallt}
  There exists a sufficiently large constant $D = D(p)$ and $\delta > 0$ such that, for all $t$ with $|t| \in (0,n^{0.55}]$,
  \[|\psi_n(t)| \leq D/|t|^{1+\delta}.\]
\end{lemma}

\begin{lemma} \label{lem:bigt} 
There exists a sufficiently large constant $D = D(p)$ such that, for all $t$ with  $|t| \in [n^{0.55},\pi \sigma_n]$, it holds that 
  \[|\psi_n(t)| \leq D/|t|^{50}.\]
\end{lemma}

We now proceed to bound~\eqref{tobound}.

  By Lemma \ref{lem:tconst},
   \[ \int\limits_{R_1} \left|\psi_n(t) - e^{-t^2/2}\right|\dt \rightarrow 0,\]
for any fixed constant $A$.
   
   For $R_2$ and $R_3$ we have the following,
   \begin{align*}
   \int\limits_{R_2 \cup R_3} \left|\psi_n(t) - e^{- t^2/2}\right|\dt & \leq
   \int\limits_{R_2 \cup R_3} \left|\psi_n(t)\right|\dt + \int\limits_{R_2 \cup R_3} \left|e^{-t^2/2}\right|\dt 
   \end{align*}
   By Lemma \ref{lem:smallt} and Lemma \ref{lem:bigt}, there exists constants $D = D(p), \delta > 0$ such that, $\left|\psi_n(t)\right| \leq \frac{D}{t^{1+\delta}}$ for all $n$ and all $t$ with $|t| \in (0,\pi \sigma_n]$. Therefore, 
   \[\int\limits_{R_2 \cup R_3} \left|\psi_n(t) - e^{- t^2/2}\right|\dt \leq \int\limits_{R_2 \cup R_3} \left|\frac{D}{t^{1 + \delta}}\right|\dt + \int\limits_{R_2 \cup R_3} \left|e^{- t^2/2}\right|\dt.\]
   Since $D/|t|^{1+\delta}$ and $e^{-t^2/2}$ both have finite integral over $(-\infty,-1) \cup (1,\infty)$, the last line above can be made smaller than any $\epsilon$ for large enough constant $A = A(\epsilon,p)$. 
   \end{proof}

\section{Proof sketch for bounding $|\psi_n(t)|$}
\label{sec:exposition}

In this section we sketch with some more detail the strategy used to bound the characteristic function 
\[\psi_n(t) \defeq \E[e^{it R_n}].\]
As a warm up, suppose that $R_n$ was the sum of $n$ i.i.d random variables $X_i$. Then, by independence, 
  $$ \psi_n(t) = \E\left[e^{it \sum\limits_{i = 1}^nX_i}\right] = \prod\limits_{i = 1}^n \E\left[e^{it X_i}\right] = \E\left[e^{it X_1}\right]^n.$$
  Thus if $\left|\E[e^{it X_1}]\right|$ is bounded sufficiently far from 1, it would follow that $|\psi_n(t)|$ is small. Of course in our case $R_n$ is the sum of dependent random variables, and one does not immediately have the expression decompose as a product. The idea that gets around this issue is to first reveal a subset $F$ of the edges and then, conditioning on the values of the edges in $F$ (assuming some nice event $\Lambda$ occurs), show that the expectation is small. For certain choices of $F$ the conditional expectation {\em does} decompose as a product, and thus the estimation becomes easier. If the good event $\Lambda$ happens with high enough probability, then one has successfully bounded $\psi_n(t)$. 
  
We now show an argument that bounds the $\psi_n(t)$ when $n^{1/2} \ll |t| \ll n$.
  For starters, suppose $F$ was all the edges of $\binom{[n]}{2}$ except for a perfect matching $M$, and let $X_F$ denote the indicator vector for the edges in $F$ that appear in $G$. For $e = \{u,v\} \in M$ let $C_e$ denote the number of paths of length 2 from $u$ to $v$ that appear in $G$ (note any such path must consist of edges in $F$). Then conditioned on the value of $X_F$, the expectation becomes 
  \[\E\left[e^{it(C + \sum\limits_{e \in M}C_eX_e)/\sigma_n}\right]\]
  where $C$ denotes the number of triangles that appear consisting only of edges in $F$. Note that $C$ and the $C_e$ are all constants conditioned on the value of $X_F$. Also, each $C_e = C_e(X_F)$ is a binomial random variable, and thus each is concentrated around $np^2$. Thus for a ``typical'' value of $X_F$ one has (roughly)
  \begin{align*} \left| \E[e^{it R_n}\mid X_F]\right| & = \left|\E[e^{it(\sum\limits_{e \in M} C_e X_e)/\sigma_n }]\right| \\
  & \approx \left|\prod\limits_{e \in M}\E[e^{it np^2 X_e /\sigma_n]}\right| \\
  & \leq \left(1-8 p (1-p) \left\|\frac{tnp^2}{2\pi \sigma_n} \right\|^2\right)^{n/2} \qquad \mbox{(applying Lemma \ref{lem:coslem})} \\
  & \approx \left(1-8 p (1-p) \left(\frac{tnp^2}{2\pi \sigma_n} \right)^2\right)^{n/2} \qquad \mbox{(since $\sigma_n = \Theta(n^2)$ and $|t| \ll n$)} \\
  & \approx \left(1-\Theta(t^2/n^2)\right)^{n/2} \\
  & \approx e^{-\Theta(t^2/n)}.
  \end{align*}  
  Thus if $|t| \gg n^{1/2}$ the above will be small.

  
  In Section \ref{sec:smallt} we push the above analysis to cover the range where $t \leq n^{.55}$. There we instead let $M$ be a bipartite subgraph obtained by taking a disjoint union of $k$ perfect matchings, where $k$ is chosen to depend on $t$. As above we first reveal all edges in $F \defeq \binom{[n]}{2} - M$ and then condition on the value of $X_F$. We then count triangles according to how many edges are in $M$, letting $C$, $Y$, and $Z$ denote the number of triangles with $0$,$1$, and $2$ edges in $M$ respectively. As before $C$ will be a constant conditioned on $X_F$, and $Y = \sum\limits_{e \in M}C_e X_e$ is the sum of $nk/2$ independent random variables. 
  
  For $k$ large enough, $\E[e^{itY/\sigma_n}]$ will be small conditioned on a ``typical'' $X_F$, and the analysis follows just as above because the expectation decomposes as a product. The difficulty now is $Z$ is a degree 2 polynomial in the variables $\{X_e: e\in M\}$, and thus $\E[e^{it(Y+Z)/\sigma_n}]$ does not decompose as a product even after conditioning on $X_F$. However, by estimating the variance of $Z$ we will find that $tZ/\sigma_n$ will be tightly concentrated in an interval of size $o(1)$, whereas $tY/\sigma_n$ will be roughly uniform mod $2\pi$. Therefore, although $Z$ has a complicated dependence on $Y$, $t(Y+Z)/\sigma_n$ will still be roughly uniform mod $2\pi$ and $|\psi_n(t)|$ will be small. It should be noted that we currently do not know how to prove a stronger bound than $|\psi_n(t)| \leq 1/t^{1+\delta}$ in this range (for contrast, the argument with a single perfect matching implies an exponentially small bound). This seems to be a major obstacle for obtaining a stronger quantitative local limit law.

The above argument is not delicate enough to deal with arbitrary $t$ with $|t| \gg n$.
In Section~\ref{sec:bigt}, we use a different conditioning argument to bound $\psi_n(t)$ for all $t$ with $|t| \in [n^{.55}, \pi\sigma_n]$.
This argument is based on partitioning ${[n] \choose 2}$ into two sets $E$ and $F$,  
and then studying the difference between the number of triangles in the two random graphs $(X_E, X_F)$ and $(X_E', X_F)$ (where $X_E'$ is an independent copy of the
random variable $X_E$).

\section{Small $|t|$}

In this section we prove Lemma~\ref{lem:tconst}.

\noindent
{\bf Lemma \ref{lem:tconst} (restated).} {\it
Let $A$ be a fixed positive real number. Then
\[ \int\limits_{-A}^{A} \left| \psi_n(t) - e^{- t^2/2}\right|\dt \rightarrow 0\]
as $n \rightarrow \infty$.}
\medskip

The proof essentially follows from the central limit theorem for triangle counts. We provide some of the details by applying a few standard results from probability theory regarding the method of moments. We begin with some additional preliminaries that we borrow (with minor changes) from Durrett's textbook ``Probability: Theory and Examples'' \cite{durrett2010probability}. 

For a random variable $X$, its {\em distribution function} is the function $F(x) \defeq \Pr[X \leq x]$. A sequence of distribution functions is said to {\em converge weakly} to a limit $F$ if $F_n(x) \rightarrow F(x)$ for all $x$ that are continuity points of $F$. A sequence of random variables $X_n$ is said to {\em converge in distribution} to a limit $X_{\infty}$ (written $X_n \xrightarrow[]{d} X_{\infty}$) if their distribution functions converge weakly. 

The moment method gives a useful sufficient condition for when a sequence of random variables converge in distribution. 
\begin{theorem}
 Let $X_n$ be a sequence of random variables. Suppose that $\E[X^k]$ has a limit $\mu_k$ for each $k$ and
 \[ \limsup\limits_{k \rightarrow \infty} \mu_{2k}^{1/2k}/2k < \infty;\]
 then $X_n \xrightarrow[]{d} X_{\infty}$ where $X_{\infty}$ is the unique distribution with the moments $\mu_k$. 
\end{theorem}

In the Appendix we provide the standard calculation that $\E[R_n^k] \rightarrow \mu_k$ for all $k$, where 
\[ \mu_k \defeq \left\{ 
  \begin{array}{l l}
    (k-1)!! & \quad \text{if $k$ is even}\\
    0 & \quad \text{if $k$ is odd}
  \end{array} \right.\]
are the moments of $N(0,1)$. It is easy to check that these moments do not grow too quickly and thus the theorem implies the well known central limit theorem for triangle counts:
\begin{equation}\label{eq:CLT}
  R_n \xrightarrow[]{d} N(0,1).
\end{equation}
  
Durrett also provides a theorem that relates convergence in distribution to pointwise convergence of characteristic functions. 
\begin{theorem}
{\bf Continuity Theorem.}  Let $X_n, 1 \leq n \leq \infty$, be random variables with characteristic functions $\phi_n$. If $X_n \xrightarrow[]{d} X_{\infty}$, then $\phi_n(t) \rightarrow \phi_{\infty}(t)$ for all $t$. 
\end{theorem}

Applying this with (\ref{eq:CLT}), we conclude that $\psi_n(t) \rightarrow e^{-t^2/2}$ for all $t$. To finish the proof of Lemma \ref{lem:tconst} we apply the dominated convergence theorem to conclude, for any fixed $A$, that

\[ \int\limits_{-A}^{A} \left| \psi_n(t) - e^{- t^2/2}\right|\dt \rightarrow 0\]
as $n \rightarrow \infty$.

\section{Intermediate $|t|$}\label{sec:smallt}

In this section, we prove Lemma~\ref{lem:smallt}.

\noindent
{\bf Lemma~\ref{lem:smallt} (restated).}\ {\it  There exists a sufficiently large constant $D = D(p)$ and $\delta > 0$ such that, for all $t$ with $|t| \in (0,n^{0.55}]$,
  \[|\psi_n(t)| = |\E[e^{itR_n}]| = |\E[e^{itS_n/\sigma_n}]| \leq D/|t|^{1+\delta}.\]
}
\medskip

Note that trivially $|\psi_n(t)| \leq 1$, and thus the lemma already holds for constant sized $t$. Thus we will assume that $t$ and $n$ are both bigger than a sufficiently large constant $D(p)$. To make the exposition simpler, we will assume $n$ is even (however, the same argument can be easily seen to apply when $n$ is odd).

We simplify notation by denoting $R_n$ by $R$, $S_n$ by $S$, and $\sigma_n$ by $\sigma$. Partition $[n]$ into sets $U,V$ both of size $n/2$ and let $P \subseteq {[n] \choose 2}$ be the complete bipartite graph between vertex sets $U,V$.  Let $k < \frac{n}{10^{10}}$ be a positive integer to be determined later. Let $M_1,\cdots, M_{k} \subseteq P$ be pairwise disjoint perfect matchings between $U$ and $V$. Let $E = M_1 \cup M_2 \cup \cdots \cup M_k$, and let $F = {[n] \choose 2} \setminus E$.

Recall that for the random graph $G \in G(n,p)$, we use $X_e$ to denote the indicator for whether edge $e$ appears in $G$. 
We also use $X_E$ and $X_F$ to denote the $\{0,1\}^E$-valued random variable $(X_e)_{e \in E} $ and the $\{0,1\}^F$-valued random variable $(X_e)_{e \in F}$ respectively. Let $C(X_F), Y(X_E, X_F)$ and $Z(X_E, X_F)$ be random variables that count the number of triangles in $G(n,p)$ which have 0,1, and 2 edges in $E$ respectively (note that, by construction of $E$, no triangle may have all 3 edges in $E$).
Thus we have  $S = C(X_F) + Y(X_E, X_F) + Z(X_E, X_F)$.

We define:
$$\zeta = \E_{X_E, X_F} [ Z(X_E, X_F) ].$$

We now work towards bounding $\left|\E[e^{itS/\sigma}|]\right|$:
 \begin{align*}
 &\left|\E[e^{itS/\sigma}]\right|  = \left|\E_{X_E, X_F}[e^{it(C(X_F) + Y(X_E, X_F) + Z(X_E, X_F))/\sigma}]\right| \\
   & = \left|\E_{X_E, X_F}[ e^{it(C(X_F) + Y(X_E, X_F) + \zeta)/\sigma} + e^{it(C(X_F) + Y(X_E, X_F) + Z(X_E, X_F))/\sigma} - e^{it(C(X_F) + Y(X_E, X_F) + \zeta)/\sigma}] \right|
   \\
   & \leq \left| \E_{X_E, X_F}[e^{it(C(X_F) + Y(X_E, X_F) + \zeta)/\sigma}]\right| +  \E_{X_E, X_F}\left[ \left| e^{it(Z(X_E, X_F))/\sigma} - e^{it\zeta/\sigma}\right|\right]
 \end{align*}
 
 We bound each of the two terms separately in the following two lemmas.
We will then use these lemmas to conclude the proof of Lemma~\ref{lem:smallt}.
 
 \begin{lemma} \label{lem:lemY}
  \[\left|\E_{X_E, X_F}[e^{it(C(X_F) + Y(X_E, X_F) + \zeta)/\sigma}]\right| \leq e^{-\Theta(t^2k/n)}.\]
 \end{lemma}
 \begin{proof}
 We bound the above expectation by revealing the edges in two stages. We first reveal $X_F$, and show that with high probability
over the choice of $X_F$, some good event occurs. We then show that whenever this good event occurs, the value of the above expectation 
over the random choice of $X_E$ is small.

Formally, using the triangle inequality we get:
\begin{align}
\left|\E_{X_E, X_F}[e^{it(C(X_F) + Y(X_E, X_F) + \zeta)/\sigma}]\right| &= 
\left|\E_{X_F}\left[e^{it(C(X_F) + \zeta)/\sigma} \cdot \E_{X_E} [ e^{it(Y(X_E, X_F))/\sigma}]\right]\right|\\
&\leq 
\E_{X_F}\left[\left| \E_{X_E} [ e^{it(Y(X_E, X_F))/\sigma}]\right| \right].\label{eq:eq1}
\end{align}

 For $e = \{u,v\} \in E$ and a vector $x_F \in \{0,1\}^F$, we let $Y_e(x_F)$ denote the number of paths of length 2 from $u$ to $v$ consisting entirely of edges
$f \in F$ for which $(x_F)_f = 1$ \footnote{This differs from the exposition in Section~\ref{sec:exposition} (where $E$ is a single perfect matching), in that some length-2 paths between $u$ and $v$ here may contain edges in $E$. We do not want to count those paths in $Y_e(x_F)$.}. In this way, for a given $x_{F}$, the random variable $Y(X_E, x_F)$ equals $\sum\limits_{e \in E} Y_e(x_F) X_e$. 
 
Define
$$L = \{ x_F \in \{0,1\}^F \mid \mbox{for some $e \in E$, $Y_e(x_F) < n p^2/2$} \}.$$
Let $\Lambda$ denote the (bad) event that $X_F \in L$.

  \begin{claim} \label{claim:c1}
$$ \Pr_{X_F} [ \Lambda ] \leq e^{-\Theta(n)}.$$
  \end{claim}
  \begin{proof}
Observe that for any given $e \in E$,
the distribution of $Y_e(X_F)$ equals $\text{Bin}(m_e,p^2)$, where $m_e$ equals the number of paths of length $2$ joining the endpoints of $e$,
and consisting entirely of edges in $F$. Also note that we have $m_e \geq n - 2k \geq n(1-1/10^9)$.

By the Chernoff bound, we have:
  \begin{align*}
\Pr[\text{Bin}(m_e, p^2) < n p^2 /2] \leq e^{-n p^2(1-p^2)/200}.
  \end{align*}

Taking a union bound over all $e \in E$, we get the claim.
  \end{proof}

Next, we show that if we condition on $\Lambda$ not occurring, then the desired expectation is small.
  \begin{claim} \label{claim:c2}
For every $x_F \in \{0,1\}^F \setminus L$,
  \[\left|\E\limits_{X_E \in \{0,1\}^E}\left[e^{itY(X_E, x_F)/\sigma} \right]\right| \leq e^{-\Theta(t^2k/n)}. \]
  \end{claim}
  \begin{proof}
  Recall that $Y(X_E, x_F) = \sum\limits_{e \in E} Y_e(x_F) X_e$.
Thus we have:
  \begin{align*}
  \left|\E\limits_{X_E}\left[e^{itY(X_E, x_F)/\sigma}  \right]\right| & = \left|\E\limits_{X_E}\left[e^{it(\sum\limits_{e \in E} Y_e(x_F)X_e)/\sigma}\right]\right| \\
   & = \left| \prod\limits_{e \in E} \E\left[e^{itY_e(x_F) X_e/\sigma}  \right] \right|  \quad\mbox{by the mutual independence of $(X_e)_{e \in E}$} \\
  & \leq  \prod\limits_{e \in E} \left(1 - 8p(1-p)\left\|\frac{tY_e(x_F)}{2 \pi \sigma}\right\|^2\right) \quad \mbox{(applying Lemma \ref{lem:coslem})} \\
  & =  \prod\limits_{e \in E} \left(1 - 8p(1-p)\cdot \left(\frac{tY_e(x_F)}{2 \pi \sigma}\right)^2\right) \quad \mbox{(since $t \leq n^{0.55}$, $Y_e(x_F) \leq n$, and $\sigma = \Theta(n^2)$)}\\
  & \leq \left(1 - 8p(1-p)\cdot \left(\frac{tnp^2}{4 \pi \sigma}\right)^2\right)^{nk/2} \quad \mbox{(since $x_F \in L$).}
  \end{align*}
  Recall that $\sigma = \sqrt{\frac{n(n-1)(n-2)(n-3)D}{2}}$ for some constant $D \leq 1$. Thus $\frac{tnp^2}{4\pi\sigma} \geq \frac{tp^2}{4\pi n}$.
Therefore we may further bound the above expression by:
  \begin{align*}
  & \leq  \left(1 - 8p(1-p)\left(\frac{tp^2}{4 \pi n}\right)^2\right)^{nk/2} \\
  & \leq e^{-\frac{t^2 p^5(1-p) k}{\pi^2 n}} \\
  & = e^{-\Theta(t^2 k/n)}
  \end{align*}
  \end{proof}
 Going back to equation (\ref{eq:eq1}) we have
 \begin{align*}
\left|\E_{X_E, X_F}[e^{it(C(X_F) + Y(X_E, X_F) + \zeta)/\sigma}]\right| &\leq \E_{X_F}\left[\left| \E_{X_E} [ e^{it(Y(X_E, X_F))/\sigma}]\right| \right]\\
&\leq \Pr_{X_F} [ X_F \in L ] + \max_{x_F \in \{0,1\}^F \setminus L} \left| \E_{X_E} [ e^{it(Y(X_E, x_F))/\sigma}]\right|\\
  & \leq e^{-\Theta(n)}  +  e^{-\Theta(t^2k/n)}  \quad \mbox{(applying claims \ref{claim:c1} and \ref{claim:c2})} \\
  & \leq e^{-\Theta(t^2k/n)}.
 \end{align*}
 
  \end{proof}

  \begin{lemma} \label{lem:lemZ}

  \[  \E_{X_E, X_F}\left[\left|e^{it(Z(X_E, X_F))/\sigma} - e^{it\zeta/\sigma}\right|\right]\leq  O\left( t^{3/2 + \delta/2}\left(\frac{k}{n}\right)^{3/2}\right) + O\left(1/t^{1+\delta}\right) \]
  \end{lemma}
  \begin{proof}
  Simplifying the expression we want to bound, we get:
  \begin{align*}
 \E_{X_E, X_F}\left[\left| e^{it(Z(X_E, X_F))/\sigma} - e^{it\zeta/\sigma}\right|\right] 
= \E_{X_E, X_F}\left[\left|e^{it ( Z(X_E, X_F) - \zeta)/\sigma} - 1\right|\right].
\end{align*}

    Thus proving the lemma reduces to proving a concentration bound: namely that $Z(X_E, X_F)$ is close to $\zeta$ with high probability.
We will bound $\Var_{X_E, X_F}[Z(X_E, X_F)]$ and apply the Chebyshev inequality. This will give the desired concentration.

Let $\Delta'$ denote the set of triangles in $K_n$ that have exactly $2$ edges in $E$.
For each $r \in \Delta'$, let $T_r(X_E, X_F)$ be the indicator for the triangle $r$ appearing in $G$. For two triangles $r,s \in \Delta'$, write $r \sim s$ if $r$ and $s$ share an edge. Note for any $r \in \Delta'$ there are at most $6k$ triangles $s \in \Delta'$ for which $r \sim s$.

We have:
   \begin{align*}
   \var_{X_E, X_F}[Z(X_E, X_F)] & = \sum\limits_{r \in \Delta'}\sum\limits_{s \in \Delta'} \cov_{X_E, X_F}[T_r(X_E, X_F), T_s(X_E, X_F)] \\
   & = \sum\limits_{r \in \Delta'}\sum\limits_{s \sim r} \cov_{X_E, X_F}[T_r(X_E, X_F),T_s (X_E, X_F)] \quad \mbox{(using independence)}\\
   & \leq |\Delta'|\cdot|6k| \\
   & \leq 6nk^3 \qquad \mbox{(since $|\Delta'| = n\binom{k}{2}$)}
   \end{align*}
   
    Applying Chebyshev's inequality with $\lambda = \sqrt{6} \cdot n^{1/2} \cdot t^{1/2 + \delta/2} \cdot k^{3/2}$ we have
  \begin{align*}
  \Pr_{X_E, X_F}[ |Z(X_E, X_F) - \zeta| > \lambda] & < \frac{\var_{X_E, X_F}[Z(X_E, X_F)]}{\lambda^2} \\
  & < 1/t^{1+\delta} 
  \end{align*}   
   Recall that $\|x\|$ denotes the distance from real number $x$ to the nearest integer. Let $\Lambda$ be the (bad) event that $|Z(X_E, X_F) - \zeta| \geq \lambda$. Using the fact that for any real number $\theta$, $|e^{i\theta}-1| \leq 2\pi \cdot \|\frac{\theta}{2\pi}\|$, we have 
  \begin{align*}
  \E_{X_E, X_F}\left[ \left| e^{it(Z(X_E, X_F)-\zeta)/\sigma} - 1 \right|\right] & \leq 2\pi \E_{X_E, X_F}\left[\left\|\frac{t(Z(X_E, X_F)-\zeta)}{2\pi \sigma}\right\|\right] \\
   & \leq  2\pi \cdot \Pr[\Lambda^c] \cdot \frac{t\lambda}{2\pi\sigma} + 2\pi \cdot \Pr[\Lambda] \cdot \frac{1}{2} \\
   & \leq  \frac{t\lambda}{\sigma} + \pi \cdot \Pr[\Lambda] \\  
  & \leq \sqrt{6} \cdot t^{3/2 + \delta/2} \cdot \frac{k^{3/2} \cdot n^{1/2}}{\sigma} + \frac{\pi}{t^{1+\delta}}\\
&\leq O\left(t^{3/2 + \delta/2} \cdot \left(\frac{k}{n}\right)^{3/2}\right)  + O\left(\frac{1}{t^{1+\delta}}\right). \quad\mbox{(since $\sigma = \Theta(n^2)$)}
  \end{align*}

This concludes the proof of Lemma~\ref{lem:lemZ}.
   \end{proof} 
   To conclude the proof of Lemma \ref{lem:smallt}, we apply Lemma \ref{lem:lemY} and Lemma \ref{lem:lemZ} to get the bound
   
   \begin{equation} \label{eq:eqfinal}
   |\E[e^{itS/\sigma}]| \leq e^{-\Theta(t^2 k/1000n)}+ O\left( t^{3/2 + \delta/2} \cdot \left(\frac{k}{n}\right)^{3/2}\right) + O\left(1/t^{1+\delta}\right)
   \end{equation}
   
   It only remains to check that $k$ may be chosen as to make the right hand side of equation (\ref{eq:eqfinal}) bounded by $O(1/t^{1+\delta})$.
Set $\delta = 0.01$, and observe that for  $\Omega(1) < t < n^{0.55}$, we have the following two relations:
   \[ \frac{n \log^2(t)}{ t^2} = O\left( \frac{n}{t^{5/3 + \delta}} \right),\]
   \[ \frac{n}{t^{5/3 + \delta}} = \omega(1).\]
Thus we may choose $k$ to be an integer satisfying:
   \[k = \Omega(n \log^2(t)/ t^2) \quad \mbox{and} \quad k = O(n/t^{5/3 + \delta}).\]
   For such a $k$ we have
   \[e^{-\Theta(t^2k/n)} \leq O(1/t^{1+\delta})\]
   and
   \[t^{3/2 + \delta/2}\left(\frac{k}{n}\right)^{3/2} = O(1/t^{1+\delta}).\]
   
   This concludes the proof of Lemma~\ref{lem:smallt}.


   \section{Big $|t|$}\label{sec:bigt}
   In this section we prove Lemma~\ref{lem:bigt}.

\noindent {\bf Lemma \ref{lem:bigt} (restated).}\ {\it There exists a sufficiently large constant $D = D(p)$ such that, for all $t$ with  $|t| \in [n^{0.55},\pi \sigma_n]$, it holds that 
  \[|\E[e^{itR_n}]| = |\E[e^{itS_n/\sigma_n}]| \leq D/|t|^{50}.\]}
\medskip

 The choice of 50 here is arbitrary, in fact the lemma will hold for any fixed constant in place of 50 (as long as $D(p)$ is chosen large enough). We only choose a large number here to remind the reader that the obstacle to a better quantitative local limit law lies in bounding $\psi_n(t)$ for $|t|$ in the range $(0,n^{.55}]$.

As in the previous section, since $n$ is fixed we simplify notation by denoting $S_n$ as $S$ and $\sigma_n$ as $\sigma$.

We will break down the proof into two different cases. Both cases will use a common framework,
which we now set up.

Let $[n] = U \cup V$ be a partition of the vertices.
Define $X_U = (X_e)_{e \in {U \choose 2}}$.
For every  $x_U \in \{0,1\}^{{U\choose 2}}$,
we will show that:
$$ \E[ e^{itS/\sigma} | X_U = x_U] \leq O\left(\frac{1}{t^{50}}\right).$$
This will imply the desired bound.

{\bf From now on, we condition on $X_U = x_U$.}

Let $E_U \subseteq {U \choose 2}$ be the induced graph on $U$:
$$ E_U = \left\lbrace \{u, u^*\} \in {U \choose 2} \mid x_{\{u, u^*\} } = 1 \right\rbrace.$$
Note that $E_U$ is determined by $x_U$ and is thus fixed.

For $u \in U$, let $A_u \in \{0,1\}^V$ denote the vector indicating the neighbors of $u$ in $V$.
Thus $A_u = (X_{\{u,v\}})_{v \in V}$.

Let $B \in \{0,1\}^{{V \choose 2}}$ denote the adjacency vector of $G|_V$.
Thus $B = \{X_e \}_{e \in {V \choose 2} }$.

Note that all the entries of the $A_u$'s and $B$ are independent $p$-biased Bernoulli random variables.
We will now express the number of triangles in $G$ in terms of the $A_u$'s and $B$ (here $\langle \cdot, \cdot \rangle$ denotes the standard inner product over $\mathbb R$) :
\begin{itemize}
\item Let $S_U$ denote the number of triangles in $G$ with all three vertices in $U$ (note that $S_U$ is determined by $x_U$ and is thus fixed).

\item  The expression $\sum\limits_{\{u,u^*\} \in E_U} \langle A_u, A_{u^*} \rangle$ counts the number of triangles in $G$ that have exactly two vertices in $U$.

\item Let $P: \{0,1\}^{V} \to \{0,1\}^{{V \choose 2}}$ denote the map
defined by:
$$ P(r)_{\{u,v\}} = r_u \cdot r_v.$$

Then $\sum\limits_{u \in U} \langle P(A_u), B \rangle$ counts the number of triangles in $G$ that have exactly two vertices in $V$. 

\item Let $Q: \{0,1\}^{{V \choose 2}} \to \mathbb N$ denote the map that sends
an adjacency vector $b$ to the number of triangles in the graph represented
by $b$ (that is the triangles whose vertices are contained in $V$).

Thus $Q(B)$ counts the number of triangles in $G$ with all three vertices in $V$.
\end{itemize}

Then we have the following expression for $S$ in terms of the $A_u$'s and $B$.
\begin{align*}
S =  S_U + \sum_{u \in U} \langle P(A_u) , B \rangle + \sum_{\{u,u^*\} \in E_U} \langle A_u, A_{u^*} \rangle +  Q(B).
\end{align*}

We now bound $\E[e^{i t S/\sigma} ]$.
\begin{align*}
|\E[e^{it S/\sigma}]|^2 &= \left|\E_{(A_u)_{u \in U}, B}\left[e^{it ( S_U + \sum_{u \in U} \langle P(A_u), B \rangle + \sum_{\{u, u^*\} \in E_U } \langle A_u, A_{u^*} \rangle + Q(B) ) / \sigma } \right] \right|^2 \\
&\leq   \E_{B}\left[ \left| e^{itQ(B)/\sigma} \cdot   \E_{(A_u)_{u \in U}} \left[ e^{it ( \langle \sum_{u\in U}  P(A_u), B \rangle + \sum_{\{u, u^*\} \in E_U } \langle A_u, A_{u^*} \rangle)/\sigma} \right] \right|^2 \right]\\
&\leq   \E_{B}\left[  \left| \E_{(A_u)_{u \in U}} \left[ e^{it  ( \langle \sum_u P(A_u), B \rangle + \sum_{\{u, u^*\} \in E_U} \langle A_u, A_{u^*} \rangle )/\sigma} \right] \right|^2 \right]\\
&=   \E_{B} \E_{(A_u)_{u \in U}} \E_{(A'_u)_{u \in U}}\left[  e^{it ( \langle \sum_u P(A_u) - P(A'_u), B \rangle + \sum_{\{u, u^*\} \in E_U} \langle A_u, A_{u^*} \rangle - 
 \sum_{\{u, u^*\} \in E_U} \langle A'_u, A'_{u^*} \rangle )/\sigma}  \right] \\
&\quad\quad\mbox{(Where for each $u \in U$, $A'_u$ is an independent copy of $A_u$)}\\
&=   \E_{(A_u)_{u \in U}} \E_{(A'_u)_{u \in U}} \left[ e^{it (\sum_{\{u, u^*\} \in E_U} \langle A_u, A_{u^*} \rangle - 
 \sum_{\{u, u^*\} \in E_U} \langle A'_u, A'_{u^*} \rangle )/\sigma} \cdot \E_{B}  \left[  e^{it  \langle \sum_u P(A_u) - P(A'_u), B \rangle/\sigma}  \right] \right]\\
&=   \E_{(A_u)_{u \in U}} \E_{(A'_u)_{u \in U}} \left[ e^{it (\sum_{\{u, u^*\} \in E_U} \langle A_u, A_{u^*} \rangle - 
 \sum_{\{u, u^*\} \in E_U} \langle A'_u, A'_{u^*} \rangle )/\sigma} \cdot \E_{B}  \left[  e^{it  \langle h_{\mathbf A, \mathbf {A'}}, B \rangle/\sigma}  \right] \right].
\end{align*}
where $\mathbf A = (A_u)_{u \in U}$, $\mathbf {A'} = (A'_u)_{u \in U}$,
and where $h_{\xx} \in {\mathbb Z}^{{V \choose 2}}$ is given by:
$$h_{\xx}  = \sum_{u\in U} (P(A_u) - P(A'_u)).$$
Observe that for each $e \in {V \choose 2}$,
$(h_{\xx})_e$ is distributed as the difference of two binomials of
the form $B(|U|, p^2)$ (but the different coordinates of $h_{\xx}$ are not independent).

Our goal is to show that with high probability over the choice of $\xx$, we have that:
$$C \defeq \left| \E_{B}  \left[  e^{it  \langle h_{\xx}, B \rangle/\sigma}  \right] \right|$$
is small in absolute value.
This will imply that $\E[e^{itS/\sigma}]$ is small, as desired.

We now achieve this goal for $|t| > n^{0.55}$ using two different arguments (to cover two different ranges of $|t|$),
instantiating the above framework with different settings of $|U|$.

\subsection{Case 1: $n^{1.001} \leq |t| < \pi \sigma$}

Suppose $n^{1.001} < |t| < \pi \sigma$. For this argument, we choose $|U| = 1$.

In this case, the coordinates of $h_{\xx}$ have the following joint distribution: Let $J \subseteq V$ be a random subset where each $v \in V$ appears independently with probability $p$. Let $J'$ be an independent copy of $J$ (think of $J$ and $J'$ as two independently chosen neighborhoods of the vertex $u$). Then  the $e$ coordinate of $h_{\xx}$ is $1$ if $e \subseteq J-J'$, $0$ if $e \subseteq J \cap J'$ or $e \subseteq J^c \cap (J')^c$, and $-1$ if $e \subseteq J'-J$. A Chernoff bound implies that with probability at least $1 - e^{-\Theta(n)}$ the symmetric difference of $J$ and $J'$ will have size at least $np(1-p)/2$. In such a case $h_{\xx}$ will have $\binom{np(1-p)/2}{2} = \Theta(n^2)$ non-zero coordinates.
From now on we assume that $\xx$ are such that this event occurs
(and we call such an $\xx$ ``good").

Then we have:
\begin{align*}
C &=  \left| \E_{B}  \left[  e^{it  \langle \sum_u h_{\xx}, B \rangle/\sigma}  \right] \right|\\
&= \left|\E_B \left[ \prod_{e \in {V \choose 2} }  e^{it (h_{\xx})_e B_e / \sigma}  \right] \right|\\
&= \left|\E_B \left[ \prod_{e \in {V \choose 2}, (h_{\xx})_e \neq 0}  e^{it (h_{\xx})_e B_e / \sigma}  \right] \right|\\
&= \left|\prod_{e \in {V \choose 2}, (h_{\xx})_e \neq 0} \E_{B_e} \left[  e^{it (h_{\xx})_e B_e / \sigma}  \right] \right|\\
&\leq \prod_{e \in {V \choose 2}, (h_{\xx})_e \neq 0} \left( 1 - 8p(1-p) \cdot \left\|\frac{t \cdot |(h_{\xx})_e|}{2\pi\sigma}\right\|^2 \right)  \quad \mbox{(by Lemma~\ref{lem:coslem})}\\
&\leq \prod_{e \in {V \choose 2}, (h_{\xx})_e \neq 0} \left( 1 - 8p(1-p)\cdot\left(\frac{t }{ 2\pi \sigma}\right)^2 \right)  \quad \mbox{(since $|(h_{\xx})_e| \in \{0, \pm1\}$ and $|t| < \pi \sigma$)}\\
&\leq e^{-\frac{2p(1-p)t^2}{\pi^2 \sigma^2} \cdot \Theta(n^2)} \quad \mbox{ since $\xx$ is good}\\ 
&\leq e^{-\Theta(t^2/n^2)} \quad \mbox{(since $\sigma = \Theta(n^2)$).}
\end{align*}

Now we use the fact that $t \geq n^{1.001}$ to conclude that $D \leq \exp(-\Theta(n^{0.002}))$.

Taking into account the probability of $\xx$ being good, we get:
$$|\E[e^{itS/\sigma}]|^2 < e^{-\Theta(n)} + e^{-\Theta(n^{0.002})} \ll \frac{1}{t^{100}},$$
as desired.

\subsection{Case 2: $n^{0.55} \leq t < n^{1.01}$}

Suppose $n^{0.55} < t < n^{1.01}$. For this argument, we choose $|U| = n/2$.

As before, we have:
\begin{align*}
C &= \left|\E_B \left[ \prod_{e \in {V \choose 2} }  e^{it (h_{\xx})_e B_e / \sigma}  \right] \right|
\end{align*}

Now for each $e \in {V \choose 2}$, the distribution of $(h_{\xx})_e$ is the difference of two binomials of the form $Bin(|U|, p^2)$.
Thus, we will typically have $(h_{\xx})_e$ around $\sqrt{|U|}$ in magnitude.

For each $e \in {V \choose 2}$, let $\Lambda_e$ be the following bad event (depending on $\xx$):
$|(h_{\xx})_e| \not\in ( |U|^{0.49}, |U|^{0.51})$.
Let $\gamma = \Pr[\Lambda_e]$. By standard concentration and anti-concentration estimates for
Binomial distributions, we have that $\gamma \leq 0.1$ (provided $n$ is sufficiently large, depending on $p$).


Let $\Lambda$ be the bad event that for more than $|V|^2/4$ choices of $e \in {V \choose 2}$,
the event $\Lambda_e$ occurs. 

\begin{lemma}
There is a constant $A$ such that for every $k$:
$$ \Pr[\Lambda] < \frac{k^{Ak}}{|V|^{k}}.$$
\end{lemma}
\begin{proof}
Let $Z_e$ be the indicator variable for the event $\Lambda_e$.
For each $e$, we have $\E[Z_e] = \gamma \leq 0.1$.

Note that if $e_1, \ldots, e_k$ are pairwise disjoint, then $Z_{e_1}, \ldots, Z_{e_k}$ are mutually independent.

Let $Z = \sum_{e \in {V\choose 2}} (Z_e - \gamma)$.
Note that $\E[Z] = 0$. We will show that $\E[Z^{2k}] \leq k^{O(k)} \cdot |V|^{3k}$.
This implies that
$$ \Pr[ \Lambda ]  \leq \Pr[Z > |V|^2/8] \leq \Pr [ Z^{2k} > (|V|^2 / 8 )^{2k} ] \leq \frac{\E[Z^{2k}]}{(|V|^2/8)^{2k}} \leq k^{O(k)} \frac{1}{|V|^k},$$
as desired.

It remains to show the claimed bound on $\E[Z^{2k}]$.
We have:
\begin{align*}
\E[Z^{2k}] &= \sum_{e_1, \ldots, e_{2k} \in {V \choose 2} }\E[ \prod_{j=1}^{2k} (Z_{e_j} - \gamma) ].
\end{align*}
We call a tuple $(e_1, \ldots, e_{2k}) \in {V \choose 2} ^{2k}$ {\em intersecting}
if for every $i \in [2k]$, there exists $j \neq i$ with $e_{j} \cap e_{i} \neq \emptyset$.
The key observation is the following: if $(e_1, \ldots, e_{2k})$ is not intersecting,
then $\E[ \prod_{j=1}^{2k} (Z_{e_j} - \gamma) ] = 0$.
To see this, suppose $(e_1, \ldots, e_{2k})$ is not intersecting because $e_i$
does not intersect any other $e_j$.
Then we have:
$$ \E[ \prod_{j=1}^{2k} (Z_{e_j} - \gamma) ] =  \E[Z_{e_i} - \gamma] \cdot \E[ \prod_{j\neq i} (Z_{e_j} - \gamma) ] = 0,$$
where the first equality follows from the independence property of the $Z_e$ mentioned above.

Thus, $\E[ Z^{2k} ] \leq \sum_{(e_1, \ldots, e_{2k}) \mbox{ intersecting} } 1 $.
We conclude the proof by counting the number of intersecting tuples $(e_1, \ldots, e_{2k})$.
Note that for every intersecting tuple $(e_1, \ldots, e_{2k})$, we have $\left|\bigcup_{j= 1}^{2k} e_j \right| \leq 3k$.
The number of intersecting tuples where every edge intersects exactly one other edge 
is $k^{\Theta(k)} n^{3k}$. Notice that every intersecting tuple that is not of this form has $\left|\bigcup_{j= 1}^{2k} e_j \right| \leq 3k-1$.
Thus the number of such intersecting tuples is at most ${ (3k)^2 \choose k } \cdot n^{3k-1} = k^{O(k)} \cdot n^{3k-1}$.
Thus $\E[Z^{2k}]$ is at most $k^{O(k)} \cdot n^{3k}$, as desired.
\end{proof}

Now suppose $\Lambda$ does not occur.
Then we can bound $C$ as follows:
\begin{align*}
C &= \left|\E_B \left[ \prod_{e \in {V \choose 2} }  e^{it (h_{\xx})_e B_e / \sigma}  \right] \right|\\
&= \prod_{e \in {V \choose 2} } \left| \E_{B_e} \left[ e^{it (h_{\xx})_e B_e / \sigma}  \right] \right|\\
&\leq \prod_{e \in {V \choose 2} } \left| \left(1 - 8p(1-p) \cdot \left\|\frac{t (h_{\xx})_e}{2\pi \sigma}\right\|^2  \right) \right| \quad \mbox{(by Lemma~\ref{lem:coslem})}\\
&\leq \prod_{e \in {V \choose 2} \mid \neg \Lambda_e } \left| \left(1 - 8p(1-p) \cdot\left\|\frac{t (h_{\xx})_e}{2\pi \sigma} \right\|^2 \right) \right|\\
&= \prod_{e \in {V \choose 2} \mid \neg \Lambda_e } \left| \left(1 - 8p(1-p) \cdot\left(\frac{t (h_{\xx})_e}{2\pi \sigma} \right)^2 \right) \right| \quad \mbox{(since $t < n^{1.001}, |(h_{\xx})_e| < |U|^{0.51}$, $|U| < n$ and $\sigma = \Omega(n^2)$)}\\
&\leq \prod_{e \in {V \choose 2} \mid \neg \Lambda_e } \left| \left(1 - 8p(1-p)\cdot  \left(\frac{t |U|^{0.49}}{2\pi \sigma} \right)^2 \right) \right| \quad \mbox{(since $|(h_{\xx})_e| \geq |U|^{0.49}$)}\\
&\leq e^{ - \frac{|V|^2}{8} \cdot 8p(1-p) \cdot  \left(\frac{t |U|^{0.49}}{2\pi \sigma} \right)^2}. \quad \mbox{(since $\Lambda$ did not occur)}
\end{align*}

Now we use the fact that $|U| = |V| = n/2$, that $\sigma = \Theta(n^2)$ and that $n^{0.55} < t$.

Thus $C \leq e^{- \Theta(n^{0.08})}$.

Thus, taking into account the probability of the bad event $\Lambda$, we get:
$$|\E[e^{itS/\sigma}]|^2 \leq  O \left( \frac{k^{O(k)}}{n^{k}} \right) + e^{-\Theta(n^{0.08})} \ll \frac{1}{t^{100}},$$
(choosing $k = 200$), as desired.

\section{Appendix}
  In this section we compute the moments of the random variable $Z_n \defeq S_n - p^3 \binom{n}{3}$. 
  
  Let $\Delta$ denote the set of $\binom{n}{3}$ triangles in $K_n$. For each $t \in \Delta$ denote $X_t$ to be the indicator of the event that all edges in $t$ appear. We write $t \sim t'$ if triangles $t$ and $t'$ share an edge. Note that if triangles $t$ and $t'$ do not share any edges, the random variables $X_t$ and $X_{t'}$ are independent and 
 \[ \E[(X_t - p^3)(X_{t'} - p^3)] = 0.\]

\begin{lemma}
Let $k$ be a positive integer. Let $C = C(p)$ be the constant $C(p) \defeq \E[(X_t - p^3)(X_{t'} - p^3)]$ where $t$ and $t'$ are any two triangles that share exactly one edge. Then if $k$ is odd
\[\E[Z_n^k] = O(n^{2k-1})\]
and if $k$ is even
\[\E[Z_n^k] = \frac{\ffact{n}{2k}C^{k/2}(k-1)!!}{2^{k/2}} + O(n^{2k-1}).\]
\end{lemma}
\begin{proof}
 We start with
 \[E[Z_n^k] = \sum\limits_{t_1 \in \Delta} \cdots \sum\limits_{t_k \in \Delta} \E\left[\prod\limits_{i = 1}^{k} (X_{t_i} - p^3)\right].\]
 We say an ordered tuple $(t_1,\cdots, t_k)$ of triangles is {\em intersecting} if for every $i$ there is a $j \neq i$ for which $t_i \sim t_j$. Note that if $(t_1,\cdots, t_k)$ is not intersecting then there is an $i$ for which the random variable $X_{t_i}$ is independent with $X_{t_j}$ for all $j \neq i$. Furthermore, for such a tuple
 \[\E\left[ \prod\limits_{i = i}^k (X_{t_i} - p^3) \right] = 0.\]
 
 We now split into cases based on the parity of $k$.
 \newline
 
 {\bf Case $k$ is even:}
 
 Given an intersecting tuple we define its {\em skeleton} to be the subgraph of $K_n$ obtained by taking the union of the triangles $t_i$. Let $H$ be a graph on $2k$ vertices that consists of $k/2$ connected components, each component being the union of two triangles sharing a single edge (although there are many such graphs $H$, note they are all isomorphic).  We say a tuple $(t_1, \cdots, t_k)$ is {\em fully paired} if its skeleton is isomorphic to $H$. We first count the number of fully paired tuples by counting the number of copies of $H$ that appear in $K_n$ times the number of fully paired tuples whose skeleton is $H$. 
 
 To count the copies of $H$, first note that 
 \[\binom{n}{4}\binom{n-4}{4}\cdots \binom{n - 2k + 4}{4}\cdot \frac{1}{(k/2)!} = \frac{\ffact{n}{2k}}{24^{k/2}(k/2)!}\]
 counts the number of ways to choose the $k/2$ connected components. Within each component there are 6 choices of the shared edge of the two triangles, after which the two triangles are determined. Thus there are
 \[ \frac{\ffact{n}{2k}6^{k/2}}{24^{k/2}(k/2)!} = \frac{\ffact{n}{2k}}{2^k(k/2)!}\]
 copies. For each copy there are $k!$ tuples whose skeleton is that copy. Thus the number of fully paired tuples is
 \[ \frac{\ffact{n}{2k}k!}{2^k(k/2)!} = \frac{\ffact{n}{2k}(k-1)!!}{2^{k/2}}.\]

 For a fully paired tuples, the expression $\E\left[\prod\limits_{i = 1}^k (X_{t_i} - p^3)\right]$ splits as a product of the expectation of each connected component (which are pairwise independent). Thus,

 \begin{equation}
 \E\left[\prod\limits_{i = 1}^k (X_{t_i} - p^3)\right] = C^{k/2}.
 \end{equation}

 We now quickly argue that the number of intersecting tuples that are not fully paired is $O(n^{2k-1})$. This follows because if a tuple is intersecting but not fully paired, than its skeleton consists of at most $2k-1$ vertices. There are $O(1)$ graphs on a given set of vertices, and given such a graph, there are $O(1)$ tuples whose skeleton is isomorphic to it ($k$ is a constant). Thus there are
 \begin{equation}\label{eq:intersecting}
 \sum\limits_{i = 3}^{2k-1} O(1)\binom{n}{i} = O(n^{2k-1})
 \end{equation}
 such intersecting graphs.
 
 We then have the following calculation. Let $P$ denote the set of fully paired tuples and $Q$ denote the set of tuples that are intersecting but not fully paired.
 \begin{align*}
 \E(Z_n^k) &  = \sum\limits_{(t_1,\cdots,t_k)} \E\left[\prod\limits_{i = 1}^k (X_{t_i} - p^3)\right] \\
 & = \sum\limits_{(t_1,\cdots,t_k) \in P} \E\left[\prod\limits_{i = 1}^k (X_{t_i} - p^3)\right] + \sum\limits_{(t_1,\cdots,t_k) \in Q} \E\left[\prod\limits_{i = 1}^k (X_{t_i} - p^3)\right] \\
 & = \frac{\ffact{n}{2k}C^{k/2}(k-1)!!}{2^{k/2}} + O(n^{2k-1}).
 \end{align*}
 
 {\bf Case $k$ is odd:}
 
Let $Q$ denote the set of intersecting tuples. Note that if $k$ is odd then there are no fully paired tuples of $k$ triangles. Therefore $|Q| = O(n^{2k-1})$ and we have the following:
\begin{align*}
 \E(Z_n^k) &  = \sum\limits_{(t_1,\cdots,t_k)} \E\left[\prod\limits_{i = 1}^k (X_{t_i} - p^3)\right] \\
 & =  \sum\limits_{(t_1,\cdots,t_k) \in Q} \E\left[\prod\limits_{i = 1}^k (X_{t_i} - p^3)\right] \\
 & = O(n^{2k-1}).
 \end{align*}
 
\end{proof}

\begin{corollary}
 Let $\sigma_n^2 \defeq \textbf{Var}[S_n]$ and let $R_n \defeq (S_n-p^3\binom{n}{3})/\sigma_n$. Then $\E[R_n^k] \rightarrow \mu_k$ for all $k$ fixed, where
\[ \mu_k = \left\{ 
  \begin{array}{l l}
    (k-1)!! & \quad \text{if $k$ is even}\\
    0 & \quad \text{if $k$ is odd}
  \end{array} \right. .\]
\end{corollary}

\section*{Acknowledgements} Thanks to Brian Garnett for helpful discussions.

\pagebreak

\bibliographystyle{alpha}
\bibliography{trianglebib}
\end{document}